\documentclass[11pt,english,reqno]{amsart}
\usepackage[T1]{fontenc}
\usepackage[latin9]{inputenc}
\usepackage{verbatim}
\usepackage{prettyref}
\usepackage{amsmath}
\usepackage{amsthm}
\usepackage{amssymb}
\usepackage[foot]{amsaddr}
\usepackage{esint}
\usepackage{float}
\usepackage{hyperref}
\usepackage{cite}
\usepackage{epstopdf}
\usepackage{color}
\usepackage{transparent}
\usepackage{subfigure}
\usepackage{pgf,tikz}
\usetikzlibrary{arrows}
\usepackage{appendix}

\makeatletter
\numberwithin{equation}{section}
\numberwithin{figure}{section}
\theoremstyle{plain}
\newtheorem{thm}{\protect\theoremname}[section]
  \theoremstyle{definition}
  
  \theoremstyle{remark}
  \newtheorem{rem}[thm]{\protect\remarkname}
  \theoremstyle{plain}
  \newtheorem{cor}[thm]{\protect\corollaryname}
      \newtheorem*{fact*}{Fact}
  \theoremstyle{plain}
  \newtheorem{lem}[thm]{\protect\lemmaname}
  \theoremstyle{definition}

  \theoremstyle{plain}
  \newtheorem{prop}[thm]{\protect\propositionname}
  \theoremstyle{remark}
  \newtheorem*{rem*}{\protect\remarkname}

\usepackage{mathrsfs,amsthm,amssymb,bbm,fullpage}

\usepackage{xcolor}
\usepackage[normalem]{ulem}

\hypersetup{colorlinks=true,pdfborder=000,  citecolor  = blue, citebordercolor= {magenta}}
\hypersetup{linkcolor=purple}
\makeatletter
\let\reftagform@=\tagform@
\def\tagform@#1{\maketag@@@{(\ignorespaces\textcolor{purple}{#1}\unskip\@@italiccorr)}}
\renewcommand{\eqref}[1]{\textup{\reftagform@{\ref{#1}}}}
\makeatother

\makeatletter

\DeclareUrlCommand\ULurl@@{%
  \def\UrlLeft{\uline\bgroup}%
  \def\UrlRight{\egroup}}
\def\ULurl@#1{\hyper@linkurl{\ULurl@@{#1}}{#1}}
\DeclareRobustCommand*\ULurl{\hyper@normalise\ULurl@}
\makeatother

\newcommand{\cA}{\mathcal{A}}

\newcommand{\cH}{\mathcal{H}}
\newcommand{\cM}{\mathscr{M}}

\newcommand{\cR}{\mathcal{R}}

\newcommand{\cQ}{\mathcal{Q}}

\newcommand{\cE}{\mathcal{E}}
\newcommand{\cF}{\mathcal{F}}

\newcommand{\R}{\mathbb{R}}

\newcommand{\N}{\mathbb{N}}

\newcommand{\E}{\mathbb{E}}

\newcommand{\indicator}[1]{\mathbbm{1}\left\{{#1}\right\}}

\newcommand{\eqdist}{\stackrel{(d)}{=}}

\newcommand{\tensor}{\otimes}

\newcommand{\supp}{\text{supp}}

\newcommand{\abs}[1]{\left\lvert#1\right\rvert}
\newcommand{\norm}[1]{\lvert\lvert#1\rvert\rvert}

\newcommand{\gibbs}[1]{\left\langle #1\right\rangle}

\newrefformat{cor}{Corollary \ref{#1}}
\newrefformat{sub}{Section \ref{#1}}
\newrefformat{exam}{Example \ref{#1}}
\newrefformat{reg}{Regime \ref{#1}}
\newrefformat{reg}{Regime \ref{#1}}
\newrefformat{ass}{Assumption \ref{#1}}
\newrefformat{rem}{Remark \ref{#1}}
\newrefformat{app}{Appendix \ref{#1}}

\makeatother

\usepackage{babel}
  \providecommand{\corollaryname}{Corollary}
  \providecommand{\definitionname}{Definition}
  \providecommand{\examplename}{Example}
  \providecommand{\lemmaname}{Lemma}
  \providecommand{\propositionname}{Proposition}
  \providecommand{\remarkname}{Remark}
\providecommand{\theoremname}{Theorem}

\newrefformat{prop}{Proposition \ref{#1}}

\begin{document}

\title{On spin distributions for generic $p$-spin models} 
\author{Antonio Auffinger}
\address[Antonio Auffinger]{Department of Mathematics, Northwestern University}
\email{tuca@northwestern.edu}
\author{Aukosh Jagannath}
\address[Aukosh Jagannath]{Department of Mathematics, Harvard University}
\email{aukosh@math.harvard.edu}


\date{\today}
\begin{abstract}
We provide an alternative formula for spin distributions of generic $p$-spin glass models. As a main application of this expression, we write
spin statistics as solutions of partial differential equations and we show that the generic $p$-spin models 
satisfy multiscale Thouless--Anderson--Palmer equations as originally predicted in the work
of M\'ezard--Virasoro \cite{MV85}.
\end{abstract}

\maketitle

\section{Introduction}
Let $H_N$ be the Hamiltonian for the mixed $p$-spin model on the discrete hypercube $\{+1,-1\}^N$, 
\begin{equation}\label{eq:ham-def}
H_N(\sigma) = \sum_{p\geq 2} \frac{\beta_p}{N^{(p-1)/2}} \sum_{1\leq i_1,\ldots,i_p \leq N} g_{i_1,\ldots,i_p} \sigma_{i_1}\cdots\sigma_{i_p},
\end{equation}
where $\{g_{i_1,\ldots,i_p}\}$ are i.i.d. standard Gaussian random variables. 
Observe that if we let 
\[\xi(x)=\sum_{p\in \mathbb{N}}\beta_p^{2}x^p,\] 
then the covariance of $H_N$ satisfies
	\begin{align*}
	\mathbb E H_N(\sigma^1)H_{N}(\sigma^2)=N\xi(R_{1,2}),
	\end{align*}
where $R_{\ell,\ell'}:=\frac{1}{N}\sum_{i=1}^N\sigma_i^\ell\sigma_i^{\ell'}$ is the normalized inner-product between $\sigma^{\ell}$ and $\sigma^{\ell'}$, $\ell, \ell' \geq 1$.  We let $G_{N}$ to be the Gibbs measure associated to $H_{N}$. In this note, we will be concerned with generic $p$-spin models, that is, those models for which the linear span of the set
$\{1\}\cup\{x^p:p\geq2,\, \beta_p\neq0\}$
is dense in $\left(C([-1,1]),\norm{\cdot}_\infty\right)$.

Generic $p$-spin models are central objects in the study of mean field spin glasses. They satisfy the Ghirlanda--Guerra identities \cite{PanchGhir10}. As a consequence, 
if we let $(\sigma^\ell)_{\ell \geq 1}$ be i.i.d. draws from $G_N$,
and consider the array of overlaps
$(R_{\ell\ell'})_{\ell,\ell'\geq 1}$,
then  it is known \cite{PanchUlt13} that this array satisfies the ultrametric structure proposed in the physics literature \cite{Mez84}.
Moreover, it can be shown (see, e.g., \cite{PanchSKBook}) that the limiting law of $R_{12}$ is given  by the Parisi measure, $\zeta$, 
 the unique minimizer of the Parisi formula  \cite{AuffChenSC15, TalPF}.

In \cite{PanchSGSD13}, a family of invariance principles, called the cavity equations, were introduced
for mixed $p$-spin models. It was shown there that if the spin array 
\begin{equation}\label{eq:spin-array}
(\sigma^\ell_i)_{1\leq i\leq N, 1 \leq \ell}.
\end{equation}
satisfies the cavity equations
then they can be uniquely characterized by their overlap distributions. It was also shown 
that  mixed $p$-spin models satisfy these cavity equations modulo a regularizing perturbation that does not affect the free energy.
In fact, it can be shown (see \prettyref{prop:cavity-equations} below)  by a standard argument that
generic models satisfy these equations without perturbations. Consequently, the 
spin distributions are characterized by $\zeta$ as well by the results of \cite{PanchSGSD13}.

Panchenko also showed that the Bolthausen--Sznitman 
invariance \cite{BoltSznit98} can be utilized to provide a formula for the distribution of spins \cite{PanchSGSD13,PanchSKBook}.
The main goal of this note is to present an alternative expression  for spin distributions of generic models  in terms of
a family of branching diffusions. This new way of describing the spin distributions provides expressions for moments of spin statistics as solutions of certain partial differential equations.
We show a few examples and applications in  \prettyref{sec:Examples}.   One of our main applications is that these spin distributions
satisfy a multi-scale generalization of the Thouless--Anderson--Palmer (TAP) equations 
similar to that suggested in \cite{MPV87} and \cite{MV85}.
This complements the authors previous work on the Thouless--Anderson--Palmer equations for generic $p$-spin models
at finite particle number \cite{AuffJag16}.

\subsection{Main results}
In this paper we assume that the reader is familiar with the theory of spin distributions. For a textbook introduction, see \cite[Chapter 4]{PanchSKBook}.
We include the relevant definitions and constructions in the Appendix for the reader's convenience.  
The starting point of our analysis is the following observation,
which says that the generic $p$-spin models satisfy the cavity equations. These equations are stated in \eqref{eq:cavity-eq}.

\begin{prop}\label{prop:cavity-equations}
Let $\nu$ be a limit of the spin array \eqref{eq:spin-array} for a generic $p$-spin model.
Then $\nu$ satisfies the cavity equations \eqref{eq:cavity-eq} for $r=0$. In particular, $\nu$ is unique. 
\end{prop}

Let $q_{*}>0$ and $U$ be a positive, ultrametric subset of the sphere of radius $\sqrt{q_*}$ in $L^{2}([0,1])$ in the sense that for any $x,y, z\in U$, we have $ (x,y) \geq 0$ and $\| x-z \| \leq \max \{\| x-y \|, \| y-z \| \}$. Define the {driving process on} $U$ to be the Gaussian process,
$B_{t}(\sigma)$, indexed by $(t,\sigma)\in[0,q_{*}]\times U$, which
is centered, a.s. continuous in time and measurable in space, with covariance
\begin{equation}\label{eq:Bromw}
\text{Cov}_{B}((t_{1},\sigma^{1}),(t_{2},\sigma^{2}))=(t_{1}\wedge t_{2})\wedge(\sigma^{1},\sigma^{2}).
\end{equation}
Put concretely, for each fixed $\sigma,$ $B_{t}(\sigma)$ is a Brownian
motion and for finitely many $(\sigma^{i})$, $(B_t(\sigma^i))$ is
a family of branching Brownian motions whose branching times are given
by the inner products between these $\sigma^{i}$. 

We then define the {cavity field process on $U$} as the
solution, $Y_{t}(\sigma)$, of the SDE 
\begin{equation}
\begin{cases}
dY_{t}(\sigma)=\sqrt{\xi''(t)}dB_{t}(\sigma)\\
Y_{0}(\sigma)=h.
\end{cases}\label{eq:cav-field-proc}
\end{equation}

Let $\zeta$ be the Parisi measure for the generic $p$-spin model. Let $u$ be the unique weak solution to the Parisi initial value problem  on $(0,1)\times\R$ ,
\begin{equation}
\begin{cases}
u_{t}+\frac{\xi''(t)}{2}\left(u_{xx}+\zeta([0,t])u_{x}^{2}\right)=0, & \\
u(1,x)=\log\cosh(x).
\end{cases}\label{eq:ParisiIVP}
\end{equation}
For the definition of weak solution in this setting and basic properties
of $u$ see \cite{JagTobSC16}. We now define the local field process,
$X_{t}(\sigma)$, to be the solution to the SDE 
\begin{equation}
\begin{cases}
dX_{t}(\sigma)=\xi''(t)\zeta([0,t])u_{x}(t,X_{t}(\sigma))dt+dY_{t}(\sigma)\\
X_{0}(\sigma)=h.
\end{cases}\label{eq:local-field-process}
\end{equation}
Finally, let the magnetization process be $M_{t}(\sigma)=u_{x}(t,X_{t}(\sigma))$.
We will show that the process $X_{q_*}(\sigma)$ is related to a re-arrangement of $Y_{q_*}(\sigma)$. 
If we view $\sigma$ as a state, then $M_{q_*}(\sigma)$ will be the magnetization of this state. The basic properties of these processes, e.g., existence, measurability, continuity, etc, 
are studied briefly in \prettyref{app:driving}.
We invite the
reader to compare their definitions to \cite[Eq. IV.51]{MPV87}
and \cite[Eq. 0.20]{BoltSznit98} (see also \cite{ArgAiz09}). We remind the reader here that the support of the asymptotic Gibbs measure
for a generic $p$-spin model is positive and ultrametric by Panchenko's ultrametricity theorem and Talagrand's positivity principle \cite{PanchSKBook}, provided we take $q_*=\sup\supp(\zeta)$.

Now, for a fixed measurable function $f$ on $L^{2}([0,1])$, write the measure $\mu_{\sigma}^{f}$, on $\{-1,1\}\times\R$ as the measure with density $p(s,y;f)$ given by
\[
p(s,y;f)\propto e^{sy}e^{-\frac{(y-f)^2}{2(\xi'(1)-\xi'(q_*))}}.
\]

Observe that 
by an application of Girsanov's theorem (see specifically
\cite[Lemma 8.3.1]{JagTobPD15}),
the measure above is equivalently described as the measure on $\{-1,1\}\times\R$
such that for any bounded measurable $\phi$, 
\begin{equation}
\int\phi \; d\mu_{\sigma}^{f}:=\E\left(\frac{\sum_{s\in\{\pm1\}}\phi(s,X_{1})e_{1}^{X_{1}s}}{2\cosh(X_{1})}\Bigg\vert X_{q_{*}}(\sigma)=f(\sigma)\right).\label{eq:pure-state}
\end{equation}
\noindent For any bounded measurable $\phi$, we let $\left\langle \phi\right\rangle _{\sigma}^{f}$,
denote its expected value with respect to $\mu_\sigma^f$.  When it is unambiguous
we omit the superscript for the boundary data. For multiple
copies, $(s_{i},y_{i})_{i=1}^{\infty}$, drawn from the product $\mu_{\sigma}^{\tensor\infty}$,
we also denote the average by $\left\langle \cdot\right\rangle _{\sigma}$.

 Let $\mu$ be a random measure on $L^{2}([0,1])$ such that
the corresponding overlap array satisfies the Ghirlanda-Guerra identities (see \prettyref{app:cavity} for the definition of these identities). Consider the law of the random
variables $(S,Y)$ defined through the relation: 
\begin{equation}
\E\left\langle \phi(S,Y)\right\rangle =\E\int\left\langle \phi\right\rangle _{\sigma}^{X}d\mu(\sigma)\label{eq:loc-av-data-x}
\end{equation}
and the random variables $(S',Y')$ defined through the relation 
\[
\E\left\langle \phi(S',Y')\right\rangle =\E\int\left\langle \phi\right\rangle _{\sigma}^{Y}\frac{\cosh(Y_{q_{*}}(\sigma))}{\int\cosh(Y_{q_{*}}(\sigma))d\mu(\sigma)}d\mu(\sigma).
\]
Let $(S_{i},Y_{i})_{i\geq 1}$ be drawn from $\left(\mu_{\sigma}^{X}\right)^{\tensor\infty}$and
$(S'_{i},Y_{i}')$ be drawn from $\left(\mu_{\sigma}^{Y}\right)^{\tensor\infty}$
where $\sigma$ is drawn from $\mu$. For i.i.d. draws $(\sigma^{\ell})_{\ell\geq1}$
from $\mu^{\tensor\infty}$, we define $(S_{i}^{\ell},Y_{i}^{\ell})$ and
$(S_{i}^{\prime \ell},Y_{i}^{\prime \ell})$ analogously.

The main result of this note is the following alternative representation for spins from cavity
invariant measures. We let $\cM_{inv}^{\xi}$ denote the space of law of exchangeable arrays with entries in $\{\pm 1\}$ that satisfy the cavity equations and the Ghirlanda-Guerra identities. 
\begin{thm}
\label{thm:spins-equivalent} We have the following.
\begin{enumerate}
\item For any generic model $\xi$ and any asymptotic Gibbs measure $\mu$, let $(\sigma_{\ell})_{\ell\geq1}$
be i.i.d. draws from $\mu$, let $(S_{i}^{\ell},Y_{i}^{\ell})$ and $(S_{i}^{\prime \ell},Y_{i}^{\prime \ell})$
be defined as above with $\sigma=\sigma^{\ell}$. Then these random variables
are equal in distribution.
\item For any measure $\nu$ in $\cM_{inv}^{\xi}$, let $(s_{i}^{\ell})$
denote the array of spins and $\mu$ denote its corresponding asymptotic
Gibbs measure. Let $(S_{i}^{\ell})$ be defined as above with $\zeta=\E\mu^{\tensor^{2}}(\left(\sigma^{1},\sigma^{2}\right)\in\cdot)$.
Then we have 
\[
(s_{i}^{\ell})\eqdist(S_{i}^{\ell}).
\]
\end{enumerate}
\end{thm}
\begin{rem}  
In \cite{PanchSGSD13,PanchSKBook}, Panchenko obtained first a description of the laws of $(s_{i}^{\ell})$ 
in a finite replica symmetry breaking regime (i.e., when $\zeta$ consists of finitely many atoms)
 using Ruelle probability cascades (see \eqref{eq:spins-RPC-formula}). By sending the number of levels of replica symmetry breaking to infinity, he obtains a formula that is valid for any generic $p$-spin \cite[Theorem 4.2]{PanchSKBook}.
This is a key step in our proof of Theorem \ref{thm:spins-equivalent}. 
At finite replica symmetry breaking, the connection to the process $X_t$ can already
be seen in \cite[pp 249-250]{BoltSznit98} as a consequence of the
Bolthausen-Sznitman invariance principle.
\end{rem}

Let us now briefly present an application of this result. Let $\nu$ be
the spin distribution for a generic model and let $\mu$ be the corresponding 
asymptotic Gibbs measure. Let $\sigma\in\supp(\mu)$ and fix $q\in[0,q_{*}]$, where $q_*=\sup\supp(\zeta)$.
Let 
\[
B(\sigma,q)=\left\{ \sigma'\in\supp(\mu):(\sigma,\sigma')\geq q\right\} 
\]
be the set of points in the support of $\mu$ that are of overlap at most $q$ with $\sigma$. Recall that
by Panchenko's ultrametricity theorem \cite{PanchUlt13}, we may decompose 
\[
\supp \; \mu =\cup_{\alpha}B(\sigma^{\alpha},q)
\]
 where this union is disjoint. If we call $W_{\alpha}=B(\sigma^{\alpha},q)$,
we can then consider the law of $(s,y)$, the spin and the cavity
field, but now conditionally on $W_{\alpha}$. That is, let $\left\langle \cdot\right\rangle _{\alpha}$
denote the conditional law $\mu(\cdot\vert W_{\alpha}).$ We then have
the following result.

\begin{thm}\label{thm:multi-TAP}
(Mezard\textendash Virasoro multiscale Thouless\textendash Anderson\textendash Palmer
equations) We have that 
\[
\left\langle s\right\rangle _{\alpha}=u_{x}\bigg(q,\left\langle y\right\rangle _{\alpha}-\int_{q}^{1}\xi''(t)\zeta([0,t])dt\cdot\left\langle s\right\rangle _{\alpha}\bigg)
\]
where again $u_{x}$ is the first spatial derivative of the Parisi
PDE corresponding to $\zeta$.
\end{thm}

\subsection*{Acknowledgements}

The authors  would like to thank Louis-Pierre Arguin, G\'erard Ben Arous, Dmitry Panchenko, and Ian Tobasco for helpful discussions. 
This research was conducted while A.A. was supported by NSF DMS-1597864 and NSF Grant CAREER DMS-1653552  
and A.J. was supported by NSF OISE-1604232.

\section{Cavity equations for generic models}\label{app:cavity-generic}

\subsection{Decomposition and regularity of mixed $p$-spin Hamiltonians}\label{app:decomposition}

In this section, we present some basic properties of mixed $p$-spin Hamiltonians. Let $1 \leq n < N$.
For $\sigma = (\sigma_{1},\ldots, \sigma_{N}) \in \Sigma_{N}$, $\rho(\sigma) = (\sigma_{n+1}, \ldots, \sigma_{N}) \in \Sigma_{N-n}$, we can write the Hamiltonian $H_{N}$
as
\begin{equation}\label{eq:ll}
H_{N}(\sigma)=\tilde H_{N}(\sigma)+\sum_{i=1}^{n}\sigma_{i}y_{N,i}(\rho)+r_{N}(\sigma).
\end{equation}
where the processes $\tilde H_{N}, y_{N,i}$ and $r_{N}$ satisfy  the following lemma.
\begin{lem}
\label{lem:decomposition-lemma} There exist centered Gaussian processes $\tilde H_{N},y_{N},r_{N}$
such that \eqref{eq:ll} holds and 
\begin{align*}
\mathbb{E}\tilde H_{N}(\sigma^{1})\tilde H_{N}(\sigma^{2})= & N\xi\left (\frac{{N-1}}{N}R_{12}\right),\\
\mathbb{E}y^i_{N}(\sigma^{1})y^j_{N}(\sigma^{2})= &\delta_{ij}( \xi'(R_{12})+o_{N}(1)),\\
\mathbb{E}r_{N}(\sigma^{1})r_{N}(\sigma^{2})= & O(N^{-1}).
\end{align*}
Furthermore, there exist positive constant $C_{1}$ and $C_{2}$ so
that with probability at least $1-e^{-C_{1}N},$
\[
\max_{\sigma\in\Sigma_{N-1}}|r_{N}(1,\sigma)-r_{N}(-1,\sigma)|\leq\frac{{C_{2}}}{\sqrt{N}},
\]
and a positive constant $C_{3}$ so that 
\begin{equation}\label{eq:Delta}
 \mathbb E \exp\bigg(2 \max_{\sigma \in \Sigma_{N-1}} |r_{N}(1,\sigma) - r_{N}(-1,\sigma)|\bigg) \leq C_{3}.
\end{equation}

\end{lem}
\begin{proof} The lemma is a standard computation on Gaussian processes. Let us focus on the case $n=2$. The general case is analagous. Furthermore,  to simplify the exposition we will consider the pure $p$-spin model. The mixed case follows by linearity. Here, we set 
\begin{align*}
\tilde H_{N}(\rho(\sigma)) &= N^{-\frac{p-1}{2}}  \sum_{2 \leq i_{1},\ldots, i_{p} \leq N} g_{i_{1}\ldots i_{p}} \sigma_{i_{1}}\ldots \sigma_{i_{p}},\\
y_{N}(\rho(\sigma)) &= N^{-\frac{p-1}{2}}  \sum_{k=1}^{p}\sum_{\stackrel{2 \leq i_{1},\ldots, i_{p} \leq N}{i_{k}=1}} g_{i_{1}\ldots i_{p}}\sigma_{i_{1}}\ldots \sigma_{i_{p}}, \quad \text{and},\\ 
r_{N}(\sigma_{1}, \rho(\sigma)) &=  N^{-\frac{p-1}{2}} \sum_{l=2}^{p} \sigma_{1}^{\ell} \sum_{2\leq i_{1}, \ldots, i_{p-\ell} \leq N} J_{i_{1}\ldots i_{p-\ell}} \sigma_{i_{1}} \ldots \sigma_{i_{p-\ell}},
\end{align*}
where $g_{i_1,\ldots,i_p}$ are as above and $J_{i_{1}\ldots i_{p-\ell}}$ are centered Gaussian random variables with variance equal to $\binom{p}{\ell}$: $J_{i_{1}\ldots i_{p-\ell}}$ is the sum of the $g_{i_{1} \ldots i_{p}}$ where the index $1$ appears exactly $\ell$ times.
Computing the variance of these three Gaussian processes give us the the first three statements of the Lemma. For the second to last  and last statement, note that for any $\sigma \in \Sigma_{N-1}$, $r(1,\sigma)-r(-1,\sigma)$ is a centered Gaussian process with variance equal to 
$$ \frac{4}{N^{p-1}} \sum_{\ell = 3, \:\ell \text{ odd}}^{p} \binom{p}{\ell} (N-1)^{p-\ell} \leq \frac{C_{p}}{N^{2}}, $$ 
for some constant $C_{p}$.
A standard application of Borell's inequality and the Sudakov-Fernique's inequality \cite{ATbook} gives us the desired result. 
\end{proof}

We now turn to the proof that generic models satisfy the cavity equations. 
The argument is fairly standard -- see for example \cite[Chapter 3, Theorem 3.6]{PanchSKBook}. 

\begin{proof}[\emph{\textbf{Proof of \prettyref{prop:cavity-equations}}}]
Fix $n$ sites and a $C_{l}$ as in \eqref{eq:cavity-eq}. By site symmetry,
we may assume that these are the last $n$ sites. Our goal is then
to show that 
\begin{equation}\label{eq:hotdog}
\E\prod_{l\leq q}\left\langle \prod_{i\in C_{l}}\sigma_{i}\right\rangle =\E\prod_{i\leq q}\frac{\left\langle \prod_{i\in C_{l}}\tanh(g_{\xi',i}(\sigma))\cE_{n}\right\rangle }{\left\langle \cE_{n}\right\rangle }+o_{N}(1).
\end{equation}

With this observation in hand, note that  by Lemma \ref{lem:decomposition-lemma},
 the left side of \eqref{eq:hotdog} is equivalent to 
\[
\E\prod_{l\leq q}\frac{\left\langle \prod_{i\in C_{l}}\tanh(y_{N,i}(\sigma))\cE_{n,0}\right\rangle _{G'}}{\left\langle \cE_{n,0}\right\rangle _{G'}^{q}},
\]
where $G'$ is the Gibbs measure for $\tilde H_{N}$ on $\Sigma_{N-n}$. 

By a localization and Stone-Weierstrass argument, we see that it suffices
to show that 
\[
\E\prod_{l\leq q}\left\langle \prod_{i\in C_{l}}\tanh(y_{i}(\sigma))\cE_{n}\right\rangle _{G'}\left\langle \cE_{n}\right\rangle _{G'}^{k}=\E\prod_{l\leq q}\left\langle \prod_{i\in C_{l}}\tanh(g_{\xi'i}(\sigma))\cE_{n}\right\rangle _{G}\left\langle \cE_{n}\right\rangle _{G}^{k}+o_{N}(1).
\]
Evidently, this will follow provided the limiting overlap distribution
for $\E G'^{\tensor\infty}$ and $\E G^{\tensor\infty}$ are the same.
As generic models are known to have a unique limiting overlap distribution (by Lemma 3.6 of \cite{PanchSKBook}),
it suffices to show that in fact the overlap distribution of law of
$H'_{N}$ and $H_{N-n}$ are the same. Observe that
\begin{align*}
\abs{\text{Cov}_{H'}(\sigma^{1},\sigma^{2})-\text{Cov}_{H}(\sigma^{1},\sigma^{2})}  =N\abs{\xi\left(\frac{N}{N+n}R_{12}\right)-\xi(R_{12})} \leq C(\xi,n),
\end{align*}
uniformly for $\sigma^{1},\sigma^{2}\in\Sigma_{N-n}$, so that by
a standard interpolation argument (see, e.g., \cite[Theorem 3.6]{PanchSKBook}) we have that  the free energy of these 
two systems is the same in the limit $N\to\infty$. An explicit differentiation argument (see \cite[Theorem 3.7]{PanchSKBook})
combined with \cite[Theorem 2.13]{PanchSKBook} shows that the overlap distributions are the same. 
\end{proof}

\section{Proofs of representation formulas}\label{sec:proofs-at-infinite}

We now turn to the proofs of the results at infinite particle number.
Before we can state these results we need to recall certain basic
results of Panchenko from the theory of spin distributions \cite{PanchSGSD13,PanchSKBook}. The notation here follows \cite[Chapter 4]{PanchSKBook} (alternatively, see the Appendix below).

\subsection{Preliminaries}

We begin with the observation that
if we apply the cavity equations, \eqref{eq:cavity-eq} with $n=m$
and $r=0$, we get that, 
\[
\E\prod_{l\leq q}\prod_{i\in C_{l}}s_{i}^{l}=\E\frac{\prod_{l\leq q}\E'\prod_{i\in C_{l}}\tanh(G_{\xi',i}(\bar{\sigma}))\prod_{i\leq n}\cosh(G_{\xi',i}(\bar{\sigma}))}{\left(\E'\prod_{i\leq n}\cosh(G_{\xi',i})\right)^{q}}.
\]
Note that the righthand side is a function of only the overlap distribution
of $\bar{\sigma}$ corresponding to $\nu$. Let the law of $R_{12}$
be denoted by $\zeta$. Suppose that $\zeta$ consists of $r+1$ atom\emph{s.
}Then, since $\mu$ satisfies the Ghirlanda-Guerra identities by assumption,
we know that this can also be written as 
\begin{equation}
\E\prod_{l\leq q}\prod_{i\in C_{l}}s_{i}^{l}=\E\frac{\prod_{l\leq q}\sum_{\alpha}w_{\alpha}\prod_{i\in C_{l}}\tanh(g_{\xi',i}(h_{\alpha}))\prod_{i\leq n}\cosh(g_{\xi',i}(\bar{\sigma}))}{\left(\sum w_{\alpha}\prod_{i\leq n}\cosh(g_{\xi',i}(\bar{\sigma}))\right)^{q}}.\label{eq:spins-RPC-formula}
\end{equation}
Here, $(w_{\alpha})_{\alpha\in\partial\cA_{r}}$ are the weights corresponding
to a $RPC(\zeta)$ and $\{h_{\alpha}\}_{\alpha\in\cA_{r}}$ are the
corresponding vectors, with $\cA_{r}=\mathbb N^{r}$ viewed as a tree with $r$ levels.

For a vertex $\alpha$ of a tree, we denote by $\abs{\alpha}$ the
depth of $\alpha$, that is its (edge or vertex) distance from the
root. We denote by $p(\alpha)$ to be the set of vertices in the path
from the root to $\alpha$. For two vertices $\alpha,\beta$, we let
$\alpha\wedge\beta$ denote their least common ancestor, and we say
that $\alpha\precsim\beta$ if $\alpha\in p(\beta)$. In particular
$\alpha\precsim\alpha$.  We say that
$\alpha\nsim\beta$ if neither $\alpha\precsim\beta$ nor $\beta\precsim\alpha$.

In this setting, it is well known that $g_{\xi'}(h_{\alpha})$ has
the following explicit version. Let $(\eta_{\alpha})_{\alpha\in\cA_{r}}$
be i.i.d. gaussians, then 
\[
g_{\xi'}(h_{\alpha})=\sum_{\beta\precsim\alpha}\eta_{\beta}\left(\xi'(q_{\abs{\beta}})-\xi'(q_{\abs{\beta}-1})\right)^{1/2}.
\]
It was then showed by Panchenko that the above also has the following
representation in terms of ``tilted'' variables $\eta'$ as follows. 

We define the following family of functions $Z_{p}:\R^{p}\to\R$ with
$0\leq p\leq r$ recursively as follows. Let 
\[
Z_{r}(x)=\log\cosh(\sum_{i=1}^{r}x_{i}(\xi'(q_{i})-\xi'(q_{i-1}))^{1/2})
\]
and let 
\begin{equation}
Z_{p}(x)=\frac{1}{\zeta([0,q_{p}])}\log\int\exp\left(\zeta([0,q_{p}])\cdot Z_{p+1}(x,z)\right)d\gamma(z)\label{eq:Z-def}
\end{equation}
where $d\gamma$ is the standard gaussian measure on $\R$. We then
define the transition kernels 
\[
K_{p}(x,dx_{p+1})=\exp\left(\zeta([0,q_{p}])\left(Z_{p+1}(x,x_{p+1})-Z_{p}(x,x_{p+1})\right)\right)d\gamma(x_{p+1}).
\]
Also, we define $\eta'_{\alpha}$ as as the random variable with
law $K_{\abs{\alpha}}((\eta_{\beta})_{\beta\precsim\alpha},\cdot)$.
Finally, define 
\[
g'_{\xi'}(h_{\alpha})=\sum_{\beta\precsim\alpha}\eta'_{\beta}\left(\xi'(q_{\abs{\beta}})-\xi'(q_{\abs{\beta-1}})\right)^{1/2}.
\]
Define $g_{\xi',i}'$ analogously. We then have the following proposition.
\begin{prop}
[Panchenko \cite{PanchSGSD13}] Let $w_{\alpha}$ be as above and let 
\[
w_{\alpha}'=\frac{w_{\alpha}\prod_{i\leq n}\cosh(g_{\xi',i}(h_{\alpha}))}{\sum w_{\alpha}\prod_{i\leq n}\cosh(g_{\xi',i}(h_{\alpha}))}
\]
Then we have 
\[
\left((w_{\alpha}',g_{\xi',i}(h_{\alpha}))\right)_{\alpha}\eqdist\left((w_{\alpha},g'_{\xi',i}(h_{\alpha}))\right)_{\alpha}.
\]
\end{prop}
If we apply this proposition to \eqref{eq:spins-RPC-formula}, we
have that 
\[
\E\prod_{l\leq q}\prod_{i\in C_{l}}s_{i}^{l}=\E\prod_{l\leq q}\sum_{\alpha}w_{\alpha}\prod_{i\in C_{l}}\tanh(g'_{\xi',i}(h_{\alpha})).
\]

\subsection{Proof of \prettyref{thm:spins-equivalent}.}

We now turn to proving \prettyref{thm:spins-equivalent}. We begin
with the following two lemmas.
\begin{lem}
\label{lem:equiv-dist} Let $h_{\alpha}$,\, $\eta_{\alpha}$, $g_{\xi},g_{\xi}^{'}$
be as above. We then have the following equalities in distribution
\begin{align*}
\left(g(h_{\alpha})\right)_{\alpha} & \eqdist\left(B_{q_{*}}(h_{\alpha})\right)_{\alpha}\\
\left(g_{\xi'}(h_{\alpha})\right)_{\alpha} & \eqdist\left(Y_{q_{*}}(h_{\alpha})\right)_{\alpha}\\
\left(g_{\xi'}'(h_{\alpha})\right)_{\alpha} & \eqdist\left(X_{q_{*}}(h_{\alpha})\right)_{\alpha}.
\end{align*}
\end{lem}
\begin{proof}
Observe by the independent increments property of Brownian motion,
we have that 
\[
(\eta_{\alpha})\eqdist(B(q_{\abs{\alpha}},h_{\alpha})-B(q_{\abs{\alpha}-1},h_{\alpha})).
\]
This yields the first two equalities. It remains to see the last equality. 

To this end, fix $h_{\alpha}$, and consider the process $X_{t}$
thats solves the SDE \prettyref{eq:local-field-process}. Then if
$Y_{t}$ is distributed like $Y$ as above with respect to some measure
$Q$ , then by Girsanov's theorem \cite[Lemma 8.3.1]{JagTobPD15},
we have that with respect to the measure $P$ with Radon-Nikodym derivative
\[
\frac{dP}{dQ}(t)=e^{\int_{0}^{t}\zeta\left([0,s]\right)du_{s}},
\]
the process $Y_{t}$ has the same law as $X_{t}$. In particular,
for the finite collection of times $q_{i}$ we have that 
\begin{align*}
\E_{P}F(X_{q_{0}}\ldots,X_{q_{r}}) & =\int F(Y_{q_{0}},\ldots,Y_{q_{r}})e^{\int_{0}^{t}\zeta du(s,Y_{s})}dQ(Y)\\
 & =\int F(Y_{t_{1}},\ldots,Y_{t_{k}})\prod_{i=0}^{k}e^{\zeta\left([0,q_{k}]\right)\left(u(q_{k},Y_{q_{k}})-u(q_{k-1},Y_{q_{k-1}})\right)}dQ.
\end{align*}
By recognizing the law of $(B_{q_{k}})$ and $(Y_{q_{k}})$ as Gaussian
random variables, and \prettyref{eq:Z-def} as the Cole-Hopf solution
of the Parisi IVP \prettyref{eq:ParisiIVP}, $u(q_{k},x)=Z_{k}(x)$,
the result follows. 
\end{proof}
We now need the following continuity theorem. This is intimately related
to continuity results commonly used in the literature, though the
method of proof is different. 

Let $\mathcal{Q}_{d}$ denote the set of $d\times d$ matrices of
the form 
\[
\mathcal{Q}_{d}=\{(q_{ij})_{i,j\in[d]}:\:q_{ij}\in[0,1],\:q_{ij}=q_{jk},\:q_{ij}\geq q_{ik}\wedge q_{kj}\:\forall i,j,k\}.
\]
Note that this set is a compact subset of $\R^{d^{2}}$. Consider
the space $\Pr([0,1])$ equipped with the weak-{*} topology. Then
the product space $\Pr([0,1])\times\mathcal{Q}_{d}$ is compact Polish.
For any $Q\in\cQ_{d}$, let $(\sigma^{i}(Q))_{i=1}^{d}\subset\cH$
be a collection of vectors whose gram-matrix is $Q$. We can then
define the functional 
\[
\mathcal{R}(\zeta,Q)=\E\prod_{i=1}^{d}u_{x}(q_{*},X_{q_{*}}(\sigma^{i})).
\]

\begin{lem}
\label{lem:R-continuous}We have that $\mathcal{R}$ is well-defined
and is jointly continuous. 
\end{lem}
\begin{proof}
Let $(\sigma^{i})_{i=1}^{d}$ be any collection with overlap matrix
$Q$. Recall the infinitesimal generator, $L^{lf}$, of the collection
$\left(X_{t}(\sigma^{i})\right)$ from \prettyref{eq:inf-gen-loc-field}.
Observe that $L^{lf}$ depends on $(\sigma^{i})$ only through their
overlap matrix, which is $Q$. Thus the law is determined by this
matrix and $\cR$ is well-defined. 

We now turn to proving continuity. As $\Pr([0,1])\times\mathcal{Q}$
is compact Polish, it suffices to show that for $\zeta_{r}\to\zeta$
and $Q^{r}=(q_{ij}^{r})$ with $q_{ij}^{r}\to q_{ij}$, $1\leq i,j\leq l$,
\[
\mathcal{R}(\zeta_{r},Q^{r})\to\mathcal{R}(\zeta,Q),
\]
 as $r\to\infty.$ 

Let $a_{ij}^{r}$ and $b_{i}^{r}$ be the coefficients of the diffusion
associated to the local field process $X^{\zeta_{r},Q^{r}}$. By \prettyref{eq:inf-gen-loc-field},
we have 
\[
a_{ij}^{r}(t)=\mathbf{1}_{\{t\leq q_{ij}^{r}\}},\quad b_{i}^{r}(t,\cdot)=\xi''\zeta_{r}u_{x}^{r}(t,\cdot),
\]
where $u^{r}$ is the solution to the Parisi initial value problem
corresponding to $\zeta_{r}$. These coefficients are all uniformly
bounded, measurable in time and smooth in space. Furthermore, $\xi$
is continuous, so that 
\begin{align}
\int_{0}^{t}\big(|a_{ij}^{r}(s)-a_{ij}(s)|+\sup_{x}&|b^{r}(s,x)-b(s,x)|\big)ds\nonumber \\
 &   \leq|q_{ij}^{r}-q_{ij}|+\int_{0}^{t}\sup_{x}|\zeta_{r}\left([0,s]\right)u_{x}^{_{r}}(t,x)-\zeta\left([0,s]\right)u_{x}(t,x)|ds\to0\label{eq:SVestimate}
\end{align}
as $r\to\infty$ since $u_{x}^{r}$ converges uniformly to $u_{x}$
by \cite[Prop. 1]{AuffChenPM15} as $\zeta_{r}\to\zeta$.

By Stroock-Varadhan's theorem \cite[Theorem 11.1.4]{stroock1979multidimensional}, 
the convergence from \prettyref{eq:SVestimate} implies that the
laws of the solutions to the corresponding martingale problems converge.
As $(x_{1},\ldots,x_{d})\mapsto\prod_{i=1}^{d}\tanh(x_{i})$ is a
continuous bounded function we obtain the continuity of $F$. 
\end{proof}
We may now turn to the proof of the main theorem of this section. 
\begin{proof}[\textbf{\emph{Proof of \prettyref{thm:spins-equivalent}}}]
 Suppose first that $\zeta$ consists of $r+1$ atoms. In this setting
the result has already been proved by the aforementioned results of
Panchenko combined with \prettyref{lem:equiv-dist}. The main task
is to prove these results for general $\zeta$. To this end, let $\zeta_{r}\to\zeta$
be atomic. Denote the spins corresponding to these measures by $s_{i,r}^{l}$. 

Correspondingly, for any collection of moments we have 
\[
\E\prod_{l\leq q}\prod_{i\in C_{l}}s_{i,r}^{l}=\E\left\langle \cR(\zeta_{r},Q)\right\rangle .
\]
Recall that the overlap distribution converges in law when $\zeta_{r}\to\zeta$,
thus by \prettyref{lem:R-continuous} and a standard argument, 
\[
\E\left\langle \cR(\zeta_{r},Q)\right\rangle _{r}\to\E\left\langle \cR(\zeta,Q)\right\rangle =\E\left\langle \prod_{i\leq q}\prod_{i\in C_{l}}\tanh(X_{q_{*}}^{i}(\sigma^{l}))\right\rangle .
\]
However, as the overlap distribution determines the spin distribution,
we see that 
\[
\E\prod_{l\leq q}\prod_{i\in C_{l}}s_{i,r}^{l}\to\E\prod_{l\leq q}\prod_{i\in C_{l}}s_{i}^{l}=\E\frac{\left\langle \prod_{l\leq q}\prod_{i\in C_{l}}\tanh(Y_{q_{*}}^{i}(\sigma^{l}))\prod_{i\leq n}\cosh(Y_{q_{*}}^{i}(\sigma^{l}))\right\rangle }{\left\langle \prod_{i\leq n}\cosh(Y_{q_{*}}^{i}(\sigma^{l}))\right\rangle ^{q}}.
\]
This yields both results.
\end{proof}

\section{Proof of  Theorem \ref{thm:multi-TAP}} We now prove that the TAP equation holds at infinite particle number.
Before stating this proof we point out two well-known \cite{AuffChenSC15,JagTobPD15} but useful facts: the magnetization process $u_x(s,X_s(\sigma))$ is a martingale for fixed $\sigma$ and $u_x(t,x)=\tanh(x)$ for $t\geq q_*=sup\text{ supp } \zeta$.

\begin{proof}[\textbf{\emph{Proof of \prettyref{thm:multi-TAP}}}]
 Consider $\left\langle s\right\rangle _{\alpha}$,
if we compute the joint moments of this expectation 
\[
\left\langle s\right\rangle _{\alpha}^{k}=u_{x}(q,X_{q}^{\sigma})^{k}
\]
for any $\sigma\in W_{\alpha}$. In fact, jointly, 
\[
\prod_{\alpha\in A}\left\langle s\right\rangle _{\alpha}^{k_{\alpha}}=\prod_{\alpha\in A}u_{x}(q,X_{q}^{\sigma^{\alpha}})^{k_{\alpha}}
\]
for $\abs{A}<\infty$. Thus in law, 
\begin{equation}\label{eq:hamburger}
\left\langle s\right\rangle _{\alpha}=u_{x}(q,X_{q}^{\sigma^{\alpha}}).
\end{equation}
By a similar argument 
\[
\left\langle y\right\rangle _{\alpha}=\E\left(X_{1}^{\sigma^{\alpha}}\vert\cF_{q}\right),
\]
where $\cF_{q}$ is the sigma algebra of $\sigma((B_{q}^{\alpha}(\sigma))_{\sigma\in\supp \mu}).$
However, 
\begin{align*}
\E\left(X_{1}^{\sigma^{\alpha}}\vert \cF_{q}\right) & =X_{q}^{\sigma^{\alpha}}+\int_{q}^{1}\xi''(s)\zeta([0,s])\E\left( u_{x}(s,X_{s}^{\sigma^{\alpha}})\vert\cF_{q}\right)ds\\
 & =X_{q}^{\sigma^{\alpha}}+\int_{q}^{1}\xi''(s)\zeta([0,s])ds\cdot u_{x}(q,X_{q}^{\sigma^{\alpha}})\\
 & =X_{q}^{\sigma^{\alpha}}+\int_{q}^{1}\xi''(s)\zeta([0,s])ds\cdot\left\langle s\right\rangle _{\alpha}
\end{align*}
where the first line is by definition, \eqref{eq:local-field-process}, of $X^\sigma$, and the
second line follows from the martingale property of the magnetization
process. Solving this for $X_{q}^{\sigma^{\alpha}}$ yields,
\[
X_{q}^{\sigma^{\alpha}}=\gibbs{y}_\alpha - \int_{q}^{1}\xi''(s)\zeta([0,s])ds\cdot\left\langle s\right\rangle _{\alpha}.
\]
Combining this with \eqref{eq:hamburger}, yields the result. 
\end{proof}

\section{Evaluation of spin statistics\label{sec:Examples}}
Using spin distributions, one can obtain
formulae for expectations of products of spins, either through the
directing function $\sigma$ or by taking limits of expressions using
Ruelle cascades. The goal of this section is to explain how one can
obtain expressions for such statistics as the solutions of certain partial differential
equations. The input required will be the overlap distribution
$\zeta(t)$. In particular, one can in principle evaluate these expression
using standard methods from PDEs or numerically. Rather than developing
a complete calculus of spin statistics, we aim to give a few illustrative
examples. 

At the heart of these calculations is the following key observation:
the magnetization process for any finite collection $(\sigma^{i})_{i=1}^{n}$
is a family of branching martingales whose independence properties
mimics that of the tree encoding of their overlap arrays. (This can
be formalized using the language of Branchingales. See \cite{Aus15}
for more on this.) In this section we focus on two examples: two spin
statistics, i.e., the overlap, and three spin statistics. One can
of course write out general formulas, however, we believe that these
two cases highlight the key ideas. In particular, the second case
is the main example in \cite{MV85}, where this is calculated using
replica theory. The reader is encouraged to compare the PDE and martingale
based discussion here with the notion of ``tree operators'' in that
paper. For the remainder of this subsection, all state measures should
be taken with boundary data $f(\sigma)=X_{q_{*}}(\sigma)$.

\subsection{Two Spin Statistics}

We first aim to study two spin statistics. As the spins take values
$\pm1$, there is only one nontrivial two spin statistic, namely $\E s_{1}^{1}s_{1}^{2}$
where the subscript denotes the site index and the superscript denotes
the replica index. Observe that by \prettyref{eq:loc-av-data-x},
we have that 
\begin{align*}
\E s_{1}^{1}s_{1}^{2} & =\E\int\left\langle s\right\rangle _{\sigma^{1}}\cdot\left\langle s\right\rangle _{\sigma^{2}}d\mu^{\tensor2}
  =\E\int\E\prod_{i=1}^{2}u_{x}(q_{*},X_{q_{*}}^{\sigma^{i}})d\mu^{\tensor2}.
\end{align*}
Observe that it suffices to compute $\E u_{x}(q_{*},X_{q_{*}}^{\sigma^{1}})u_{x}(q_{*},X_{q_{*}}^{\sigma^{2}}).$
There are a few natural ways to compute this. Let $q_{12}=(\sigma^{1},\sigma^{2})$.
One method is to observe that if $\Phi=\Phi_{q_{12}}$ solves 
\[
\begin{cases}
(\partial_{t}+L_{t}^{lf})\Phi=0 & [0,1]\times\R^{2}\\
\Phi(1,x,y)=\tanh(x)\tanh(y)
\end{cases},
\]
where $L^{lf}$ is the infinitesimal generator for the local field
process (see \prettyref{eq:inf-gen-loc-field}), then 
\[
\E u_{x}(q_{*},X_{q_{*}}^{\sigma^{1}})u_{x}(q_{*},X_{q_{*}}^{\sigma^{2}})=\Phi_{q_{12}}(0,h).
\]
One can study this problem using PDE methods or Ito's lemma. This
yields the expression 
\[
\E s_{1}^{1}s_{1}^{2}=\int\Phi_{s}(0,h)d\zeta(s).
\]
Alternatively, note that, by the branching martingale property of
the magnetization process, we have that
\[
\E u_{x}(q_{*},X_{q_{*}}^{\sigma^{1}})u_{x}(q_{*},X_{q_{*}}^{\sigma^{2}})=\E u_{x}(q_{12},X_{q_{12}})^{2},
\]
yielding the alternative expression 
\[
\E s_{1}^{1}s_{1}^{2}=\int\E u_{x}^{2}(s,X_{s})d\zeta(s).
\]
In the case that $\zeta$ is the Parisi measure for $\xi$ (for this
notation see \cite{JagTobPD15,AuffChenPM15}), it is well-known that
on the support of $\zeta$, 
\[
\E u_{x}^{2}(s,X_{s})=s
\]
 so that 
\[
\E s_{1}^{1}s_{1}^{2}=\int sd\zeta(s).
\]
This resolves a question from \cite[Remark 5.5]{BoltSznit98}.

\subsection{Three spin statistics.}

We now turn to computing more complicated statistics. We focus on
the case of the three spin statistic, 
$\E s_{1}^{1}s_{1}^{2}s_{1}^{3},$
as we believe this to be illustrative of the essential ideas and it is the main example give in the paper of M\'ezard-Virasoro \cite{MV85}. 

We say a function $f:[0,1]^{k}\to\R$ is \emph{symmetric} if for every
$\pi\in S_{k}$, 
\[
f(x_{\pi(1)},\ldots,x_{\pi(k)})=f(x_{1},\ldots,x_{k})
\]
In the following, we denote by $dQ(R^{n})$ the law of the overlap
array $R^{n}=(R_{ij})_{ij\in[n]}$. We say that such a function has
\emph{vanishing diagonal} if $f(x,\ldots,x)=0$. We will always assume
that $Q$ satisfies the Ghirlanda-Guerra identities.
Our goal is to prove the following:
\begin{thm}
\label{thm:three-spin-calc}We have that 
\begin{align*}
\E s_{1}^{1}s_{1}^{2}s_{1}^{3} & =\frac{3}{4}\int\int\E u_{x}(b\vee a,X_{b\vee a})^{2}u_{x}(a\wedge a,X_{a\wedge a})d\zeta(a)d\zeta(b).
\end{align*}
\end{thm}
As a starting point, again observe that from the properties of state
measures, \prettyref{eq:pure-state},
\begin{align*}
\E s_{1}^{1}s_{1}^{2}s_{1}^{3} & =\E\int\left\langle s\right\rangle _{\sigma^{1}}\cdot\left\langle s\right\rangle _{\sigma^{2}}\cdot\left\langle s\right\rangle _{\sigma^{3}}d\mu^{\tensor3}.
\end{align*}
Denote the integrand by 
\[
\cR(\sigma^{1},\sigma^{2},\sigma^{3})=\left\langle s\right\rangle _{\sigma^{1}}\cdot\left\langle s\right\rangle _{\sigma^{2}}\cdot\left\langle s\right\rangle _{\sigma^{3}}.
\]
 The proof of this result will follow from the following two lemmas.
\begin{lem}
\label{lem:obvious}We have the following.

\begin{enumerate}
\item Suppose that $g(x,y)$ is a continuous, symmetric function. Then 
\[
\int g(R_{12},R_{13})dQ=\frac{1}{2}\int\int g(x,y)d\zeta(x)d\zeta(y)+\int g(x,x)d\zeta(x).
\]
\item Suppose that $f(x,y,z)$ is a continuous symmetric function with vanishing
diagonal. Then 
\[
\int f(R_{12},R_{13},R_{23})dQ=\frac{3}{2}\int f(R_{12}\vee R_{13},R_{12}\wedge R_{13},R_{12}\wedge R_{13})dQ.
\]
\item Suppose that $f(x,y,z)$ is as above and such that $h(x,y)=f(x\vee y,x\wedge y,x\wedge y)$
is continuous. Then 
\[
\int f(R_{12},R_{13},R_{23})dQ=\frac{3}{4}\int\int h(x,y)d\zeta(x)d\zeta(y).
\]
\end{enumerate}
\end{lem}
\begin{proof}
The first claim follows immediately from the Ghirlanda--Guerra identities.
The last item is implied by the first two. It remains to prove the second
claim. By symmetry of $f$ and ultrametricity we have that
\begin{align*}
\int f(R_{12},R_{13},R_{23})dQ & =3\int_{R_{12}>R_{13}}f(R_{12},R_{13},R_{13})dQ+\int_{R_{12}=R_{13}=R_{23}}f(R_{12},R_{12},R_{12})dQ
\end{align*}
The second term is zero by the vanishing diagonal property of $f$,
so that,
\[
RHS=3\int_{R_{12}\geq R_{13}}f(R_{12},R_{13},R_{13})dQ=\frac{3}{2}\int h(R_{12},R_{13})dQ,
\]
using again the vanishing diagonal property and the definition of $h$. 
\end{proof}
\begin{lem}
\label{lem:Three-spin-function-lem}There is a continuous, symmetric
function of three variables defined on the set of ultrametric $[0,1]^{3}$
such that the function $\cR(\sigma^{1},\sigma^{2},\sigma^{3})=f(R_{12},R_{13},R_{23})$.
This function has vanishing diagonal, and satisfies 
\begin{equation}
f(a,b,b)=\E u_{x}(b,X_{b})^{2}u_{x}(a,X_{a})\label{eq:three-spin-function}
\end{equation}
 for $a\leq b$.
\end{lem}
\begin{rem}
This is to be compared with \cite[Eq. 34]{MV85}.
\end{rem}
\begin{proof}
That it is a continuous, symmetric function of the overlaps is obvious.
It suffices to show \prettyref{eq:three-spin-function}. To this end,
observe that without loss of generality $R_{12}\geq R_{13}=R_{23}$.
In this case, denoting $R_{12}=b$ and $R_{23}=R_{13}=a$, we have
that
\begin{align*}
\cR(\sigma^{1},\sigma^{2},\sigma^{3}) & =\E u_{x}(1,X^{\sigma^{1}})u_{x}(1,X^{\sigma^{2}})u_{x}(1,X^{\sigma^{3}})\\
 & =\E u_{x}(b,X_{b}^{1})u_{x}(b,X_{b}^{2})u_{x}(b,X_{b}^{3})\\
 & =\E u_{x}(b,X_{b}^{1})^{2}u_{x}(b,X_{b}^{3})\\
 & =\E u_{x}(b,X_{b}^{1})^{2}u_{x}(a,X_{a}^{3})\\
 & =\E u_{x}(b,X_{b})^{2}u_{x}(a,X_{a}).
\end{align*}
In the second line, we used independence and the martingale property.
In the third line we used that the driving processes are identical in
distribution until that time. In the fourth line we use the martingale
property and independence of local fields again. The final result
comes from the fact that the driving process for the three spins is
equivalent until $a$. 
\end{proof}
We can now prove the main result of this subsection:
\begin{proof}[\textbf{\emph{Proof of \prettyref{thm:three-spin-calc}}}]
 Recall that 
\[
\E s_{1}^{1}s_{1}^{2}s_{3}^{3}=\E\left\langle \E\cR(\sigma^{1},\sigma^{2},\sigma^{3})\right\rangle .
\]
The result then follow by combining \prettyref{lem:Three-spin-function-lem}
and part 3. of \prettyref{lem:obvious}.
\end{proof}

\appendices
\section{Appendix\label{sec:Appendix}}
\subsection{On the driving process and its descendants}\label{app:driving}

We record here the following basic properties of the driving process,
cavity field process, local field process, and magnetization process.
\begin{lem}
Let $U$ be a positive ultrametric subset of a separable Hilbert space
that is weakly closed and norm bounded equipped with the restriction
of the Borel sigma algebra. Let $B_{t}(\sigma)$ be the process defined
in \eqref{eq:Bromw}. We have the following:
\begin{enumerate}
\item The covariance structure is positive semi-definite.
\item There is a version of this process that is jointly measurable and
continuous in time.
\item For each $\sigma,$ $B_{t}(\sigma)$ has the law of a brownian motion
so that stochastic integration with respect to $B_{t}(\sigma)$ is
well-defined. 
\end{enumerate}
\end{lem}
\begin{proof}
We begin with the first. To see this, simply observe that if $\alpha_{i}\in\R$,
$(t_{i},\sigma_{i})$ are finitely many points in $[0,q_{*}]\times U$
and $\sigma_{*}\in U$ , then 
\begin{align*}
\sum\alpha_{i}\alpha_{j}\left(t_{i}\wedge t_{j}\wedge(\sigma_{i},\sigma_{j})\right) & =\sum\alpha_{i}\alpha_{j}\int\indicator{s\leq t_{i}}\indicator{s\leq t_{j}}\indicator{s\leq(\sigma_{i},\sigma_{j})}ds\\
 & \geq\sum\alpha_{i}\alpha_{j}\int\indicator{s\leq t_{i}}\indicator{s\leq t_{j}}\indicator{s\leq(\sigma_{i},\sigma_{*})}\indicator{s\leq(\sigma_{j},\sigma_{*})}ds\\
 & =\norm{\sum\alpha_{i}\indicator{s\leq t_{i}\wedge(\sigma_{i},\sigma_{*})}}_{L^{2}}\geq0.
\end{align*}
We now turn to the second. Observe first that, since $[0,q_{*}]\times U$
is separable and $\R$ is locally compact, $B_{t}(\sigma)$ has a
separable version. Furthermore, observe that $B_{t}(\sigma)$ is stochastically
continuous in norm, that is as $(t,\sigma)\to(t_{0},\sigma_{0})$
in the norm topology, $P(\abs{B_{t}(\sigma)-B_{t_{0}}(\sigma_{0})}>\epsilon)\to0$.
Thus since $U$ is weakly-closed and norm bounded it is compact in
the weak topology. Thus it has a version that is jointly measurable
by \cite[Theorem IV.4.1]{GikhmanSkorokhodVol1}. Note then, since
the covariance of $B_{t}(\sigma)$ for fixed $\sigma$ is that of
Brownian motion and $B_{t}(\sigma)$ is separable, it is in fact continuous
by \cite[Theorem IV.5.2]{GikhmanSkorokhodVol1}.

The third property was implicit in the proof of the second.
\end{proof}
We now observe the following consequence of the above proposition:
\begin{cor}
Let $U$ be a positive ultrametric subset of a separable Hilbert space
that is weakly closed and norm bounded. Then the cavity field process,
$Y_{t}(\sigma)$, the local field process, $X_{t}(\sigma)$, and the magnetization process, $M_{t}(\sigma)$, exist, are continuous
in time and Borel measurable in $\sigma.$
\end{cor}
In the above, the following observation regarding the infinitesimal
generator of the above processes will be of interest.
\begin{lem}
Let $(\sigma^{i})_{i=1}^{n}\subset U$ where $U$ is as above. Then
we have the following.
\begin{enumerate}
\item The driving process satisfies the bracket relation 
\[
\left\langle B(\sigma^{1}),B(\sigma^{2})\right\rangle _{t}=\begin{cases}
t & t\leq(\sigma,\sigma')\\
0 & t>(\sigma,\sigma')
\end{cases}.
\]
\item The cavity field process satisfies the bracket relation 
\[
\left\langle Y(\sigma^{1}),Y(\sigma^{2})\right\rangle _{t}=\begin{cases}
\xi'(t) & t\leq(\sigma^{1},\sigma^{2})\\
0 & else
\end{cases}.
\]
\item The local fields process satisfies the bracket relation 
\[
\left\langle X(\sigma^{1}),X(\sigma^{2})\right\rangle _{t}=\begin{cases}
\xi'(t) & t\leq(\sigma^{1},\sigma^{2})\\
0 & else
\end{cases}
\]
and has infinitesimal generator 
\begin{equation}
L_{t}^{lf}=\frac{\xi''(t)}{2}\left(\sum a_{ij}(t)\partial_{i}\partial_{j}+2\sum b_{i}(t,x)\partial_{i}\right)\label{eq:inf-gen-loc-field}
\end{equation}
where $a_{ij}(t)=\indicator{t\leq(\sigma^{i},\sigma^{j})}$ and $b_{i}(t,x)=\zeta([0,t])\cdot u_{x}(t,x).$
\end{enumerate}
\end{lem}
\begin{proof}
We begin with the first claim. To see this, observe that by construction,
\[
B_{t}(\sigma^{1})=B_{t}(\sigma^{2})
\]
for $t\leq(\sigma^{1},\sigma^{2})$, thus the bracket above is just
the bracket for Brownian motion. If $t>(\sigma^{1},\sigma^{2}):=q$,
then the increments $B_{t}(\sigma^{1})-B_{q}(\sigma^{1})$ and $B_{t}(\sigma^{2})-B_{q}(\sigma^{2})$
are independent Brownian motions. This yields the second regime. By
elementary properties of It\^o processes, we obtain the brackets
for $Y_{t}$ and $X_{t}$ from this argument. It remains to obtain
the infinitesimal generator for the local fields process. 

To this end, observe that if $f=f(t,x_{1},\ldots,x_{k})$ is a test
function, then It\^o's lemma applied to the process $(X_{t}(\sigma^{i}))_{i=1}^{n}$
yields 
\begin{align*}
df & =\partial_{t}f\cdot dt+\sum_{i}\partial_{x_{i}}f\cdot dX_{t}(\sigma^{i})+\frac{1}{2}\cdot\sum\partial_{x_{i}}\partial_{x_{j}}f\cdot d\left\langle X_{t}(\sigma^{i}),X_{t}(\sigma^{j})\right\rangle \\
 & =\left(\partial_{t}f+\sum_{i}\partial_{x_{i}}f\cdot\left(\xi''(t)\zeta(t)u_{x}(t,X_{t}(\sigma^{i})\right)+\frac{\xi''}{2}\sum\indicator{t\leq(\sigma^{i},\sigma^{j})}\partial_{x_{i}}\partial_{x_{j}}f\right)dt+dMart
\end{align*}
where $dMart$ is the increment for some martingale. Taking expectations
and limits in the usual fashion then yields the result. 
\end{proof}

\subsection{The Cavity Equations and Ghirlanda-Guerra Identities}\label{app:cavity}
In this section, we recall some definitions for completeness. For a textbook presentation, see \cite[Chapters 2 and 4]{PanchSKBook}.
Let $\cM$ be the set of all
measures on the set $\{-1,1\}^{\N\times\N}$ that are exchangeable, that is,
if $(s_i^\ell)$ has law $\nu\in\cM$, then
\[
(s_{\pi(i)}^{\rho(\ell)})\eqdist (s_{i}^{\ell})
\]
for any permutations $\pi,\rho$ of the natural numbers.
The Aldous-Hoover theorem \cite{AldExch83,Hov82},
states that 
if $(s_{i}^{\ell})$ is the random variable induced by some
measure $\nu\in\cM$, then there is a measurable function of four
variables, $\sigma(w,u,v,x)$, such that  
\[
(s_{i}^{\ell})\eqdist(\sigma(w,u_{\ell},v_{i},x_{\ell i}))
\]
where $w,u_{\ell},v_{i},x_{\ell i}$ are i.i.d. uniform $[0,1]$ random
variables. We call this function a \emph{directing function} for $\nu$. 
The variables $s_{i}^{\ell}$ are called the spins sampled from $\nu$.

For any $\nu$ in $\cM$ with directing function $\sigma$,
let $\bar{\sigma}(w,u,v)=\int\sigma(w,u,v,x)dx$. Note that since $\sigma$ is $\{\pm1\}$-valued, 
this encodes all of the information of $\sigma(w,u,v,\cdot )$.
Define the measure $\mu$ \,on the Hilbert space, $\cH=L^{2}([0,1],dv)$,
by the push-forward of $du$ through the map $u\mapsto\bar{\sigma}(w,u,\cdot)$,
\[
\mu=(u\mapsto\bar{\sigma}(w,u,\cdot))_{*}du.
\]
The measure $\mu$ is called the asymptotic Gibbs measure corresponding to $\nu$.

A measure $\nu$ in $\cM$ is said to satisfy the Ghirlanda-Guerra identities if the
law of  the overlap array satisfies the following property: 
for every $f\in C([-1,1]^{n})$ and $g\in C([-1,1])$, we have
\begin{equation}
\E\left\langle f(R^{n})\cdot g(R_{1,n+1})\right\rangle =\frac{1}{n}\left[\E\left\langle f(R^{n})\right\rangle \cdot\E\left\langle g(R_{12})\right\rangle +\sum_{k=2}^{n}\E\left\langle f(R^{n})\cdot g(R_{1k})\right\rangle \right],\label{eq:GGI}
\end{equation}
where by the bracket, $\left\langle \cdot\right\rangle $, we mean
integration against the relevant products of $\mu$ with itself.

A measure $\nu$  is said to satisfy the cavity equations if the following is true. 
 Fix the directing function $\sigma$ and $\bar{\sigma}$ as above. 
Let $g_{\xi'}(\bar{\sigma})$ denote the centered Gaussian process indexed by
$L^2([0,1],dv)$ with covariance
\[
\E \bigg[g_{\xi'}\big(\bar{\sigma}(w,u,\cdot)\big)g_{\xi'}\big(\bar{\sigma}(w,u',\cdot)\big)\bigg] = \xi'\bigg(\int\bar{\sigma}(w,u,v)\bar{\sigma}(w,u',v)dv\bigg)
\]
and let $G_\xi'(\bar{\sigma})=g_{\xi'}(\bar{\sigma})+z(\xi'(1)-\xi'(\norm{\bar{\sigma}(w,u,\cdot)}_{L^2(dv)}^2))^{1/2}$. 
Let $g_{\xi',i}$ and $G_{\xi',i}$ be independent copies of these processes. Let $n,m,q,r,l\geq1$ be such that $n\leq m$ and $l\leq q$.
Let $C_l \subset [m]$ and let $C_l^1=C_l\cap[n]$ and $C^2_l=C_l\cap(n+[m])$.  Let 
\[
U_{l}=\int\E'\prod_{i\in C_{l}^{1}}\tanh G_{\xi',i}(\bar{\sigma}(w,u,\cdot)\prod_{i\in C_{l}^{2}}\bar{\sigma}_{i}\cE_{n,r}du
\]
where $\E'$ is expectation in $z$,  $\bar{\sigma}_{i}=\bar{\sigma}(w,u,v_{i}),\theta(t)=\xi'(t)t-\xi(t)$,
and where 
\[
\cE_{n,r}=\exp\left(\sum_{i\leq n}\log\cosh(G_{\xi',i}(\bar{\sigma}(w,u,\cdot))+\sum_{k\leq r}G_{\theta,k}(\bar{\sigma}(w,u,\cdot))\right).
\]
Let $V=\E'\cE_{n,r}$.  The \emph{cavity equations}  for $n,m,q,r\geq1$ are then given by
\begin{equation}
\E\prod_{l\leq q}\E'\prod_{i\in C_{l}}\bar{\sigma}_{i}=\E\frac{\prod_{l\leq q}U_{l}}{V^{q}}.\label{eq:cavity-eq}
\end{equation}

\bibliographystyle{plain}
\bibliography{localfields}
\end{document}